\documentclass[oneside,A4,11pt]{amsart}
\usepackage{enumitem}
\usepackage{graphicx}
\usepackage[margin=1.4in]{geometry} 

\usepackage{amssymb}
\usepackage{amsthm}

\usepackage{amsmath}
\usepackage{color}
\usepackage[utf8]{inputenc}

\usepackage{tikz}
\usepackage[all]{xy}
\usepackage[bookmarksnumbered,colorlinks]{hyperref}
\usepackage{dsfont}
\usepackage[normalem]{ulem}
\usepackage{float} 
\def\br#1\er{\textcolor{red}{#1}} 
\newcommand{\Hm}{\overrightarrow H}

\usepackage{autonum}

\theoremstyle{definition}
\newtheorem{definition}{Definition}[section]

\begin{document}
\title{On complete trapped submanifolds in Globally Hyperbolic spacetimes}
\author[A.L. Albujer]{Alma L. Albujer} \address{Departamento de
  Matemáticas, Edificio Albert Einstein\hfill\break\indent Universidad
  de Córdoba, Campus de Rabanales,\hfill\break\indent 14071 Córdoba,
  Spain}

\author[J. Herrera]{J\'onatan Herrera} \address{Departamento de
  Matemáticas, Edificio Albert Einstein\hfill\break\indent Universidad
  de Córdoba, Campus de Rabanales,\hfill\break\indent 14071 Córdoba,
  Spain}

\author[R. Rubio]{Rafael M. Rubio} \address{Departamento de
  Matemáticas, Edificio Albert Einstein\hfill\break\indent Universidad
  de Córdoba, Campus de Rabanales,\hfill\break\indent 14071 Córdoba,
  Spain}

\newtheorem{thm}{Theorem}[section]
\newtheorem{theorem}{Theorem}[section]
\newtheorem{proposition}[thm]{Proposition} \newtheorem{lemma}[thm]{Lemma}
\newtheorem{corollary}[thm]{Corollary} \newtheorem{conv}[thm]{Convention}
\theoremstyle{definition} \newtheorem{defi}[thm]{Definition}
\newtheorem{notation}[thm]{Notation} \newtheorem{exe}[thm]{Example}
\newtheorem{conjecture}[thm]{Conjecture} \newtheorem{prob}[thm]{Problem}
\newtheorem{remark}[thm]{Remark}
\newtheorem{example}[thm]{Example}
\newcommand{\nablat}{\overline{\nabla}}
\renewcommand{\div}{\mathrm{div}}
\begin{abstract}
	 {The aim of this manuscript is to obtain rigidity and non-existence results for parabolic spacelike submanifolds with causal mean curvature vector field in orthogonally splitted spacetimes, and in particular, in globally hyperbolic spacetimes. We also obtain results regarding the geometry of submanifolds by ensuring, under some mild hypothesis, the non-existence of local minima or maxima of certain distinguished function. Furthermore, in this last case the submanifold does not need to be parabolic or even complete.}
	
	{As an application in General Relativity, we obtain several nice results regarding (non-necessarily closed) trapped surfaces in a huge family of spacetimes. In fact, we show how our technique allows us to recover some relevant previous results for trapped surfaces in both, standard static spacetimes and Generalized Robertson-Walker spacetimes.} 
\end{abstract}


\maketitle
\usetikzlibrary{matrix}

\section{Introduction}

{In this work, we examine submanifolds with causal mean curvature vector in a broad class of Lorentzian manifolds, known as \emph{orthogonally-splitted} spacetimes.}

On the one hand, orthogonally-splitted spacetimes (see \eqref{eq:1}) encompass several important families of spacetimes including, as noted by the classic result of Bernal and Sánchez \cite{BS05}, any globally hyperbolic spacetime. However, not all orthogonally-splitted spacetimes should be globally hyperbolic. For instance, standard static spacetimes also fall under this category, as well as Generalized Robertson-Walker (GRW for short) spacetimes or, more generally, the multiwarped spacetimes. Therefore, a wide variety of well-known cosmological models, such as Friedmann-Lem\^aitre-Robertson-Walker spacetimes, (Anti-)de Sitter Schwarszchild spacetimes and Kasner models (among others), can be seen as particular cases of orthogonally-splitted spacetimes.

{On the other hand, spacelike submanifolds} with causal mean curvature vector have not only immense geometric interest, but an indisputable importance from a physics standpoint. So, for example, trapped surfaces
are nowadays a very intense
field of study. As it is said in \cite{MarsSeno}, their applications in General Relativity are ubiquitous: 

\begin{quote} {\it
		To mention just a few outstanding situtations where it [the concept of trapped surface] has been essential, we can cite the development of the singularity theorems, the general analysis of gravitational collapse and formation of trapped surfaces and black holes, the study of the cosmic censorship hypothesis and the related Penrose inequality, or the numerical analysis of the Cauchy development of apparently innocuous initial data}. 
	
\end{quote}

The concept of closed (compact without boundary) trapped surfaces, was originally introduced by Penrose in 1965 for the case of 2-dimensional spacelike surfaces embedded in $4$-dimensional spacetimes in terms of the signs or the vanishing of the so-called null expansions (see \cite{HP,Pe}), and it has remained that way for many years. Note that Penrose's definition implies in particular that a trapped surface is a topological sphere and consequently compactness is required.
{With that notion of trapped surface, Hawking and Penrose were able to show the existence of singularities (i.e, of incomplete timelike geodesics) in the evolution of a Cauchy initial data set containing a trapped surface (under some additional and natural physical assumptions, see \cite{HE} for details).}

In recent years, it has become clear that this concept is related to the causal orientation of the mean curvature vector field of the submanifold, which provides a better characterization of trapped surfaces and allows their generalization to 2-codimensional spacelike submanifolds of arbitrary dimension $n$ (see \cite{Sen} for more details).  {According to this modern point of view, and following~\cite{MarsSeno}}, we recall that an embedded 2-codimensional spacelike submanifold $S$  {in a causally orientable spacetime} is called a \emph{future} (resp. \emph{past}) \emph{trapped surface} if its mean curvature vector field $\vec{H}$ is timelike and future pointing (resp. past pointing)  {all over $S$}. More recently, the notion of trapped surface has been expanded in order to include the cases where the mean curvature vector is also causal, or even zero, introducing the notions of nearly, weakly, marginally or extremal future (past) trapped surface (see Definition \ref{def:main:2}). {These more recent notions have become relevant in recent years, so for example, marginally trapped surfaces play an important role in the study of the weak cosmic censorship conjeture (see \cite{AMS2008,AM}). Also, cases with arbitrary dimension and codimension have been considered recently obtainig interesting results (see for example \cite{GS}).}

There are many {others} recent works regarding trapped surfaces in a broad way. Such works consider several backgrounds and deal with different aspects of the surfaces: geometrical properties, rigidity, representation, non-existence results, classification in some spacetimes, causality properties, etc (see for instance \cite{AGM,ACC,IB,BS05,Cr,Fl,Mu} and references therein).

\medskip

This article is organized as follows.  {In Section~\ref{sec:orthogospace} we consider the family of orthogonally splitted spacetimes and study their expanding/contracting behaviour. In Section~\ref{sec:submanifolds} the basic preliminaries on spacelike submanifolds of arbitrary codimension immersed in an orthogonally splitted spacetime are presented. In particular, we introduce a distinguished function, obtained from a global temporal one, defined over such submanifolds. Moreover, a expression for its Laplacian its given, which will be essential for obtaining our main results. Section 4 is devoted to the concept of parabolicity which can be understood, at least from the geometric analysis perspective, as an extension of the notion of compactness. In fact, although compactness is a topological property, it entails several geometrical properties in Riemannian manifolds. Such properties are also partially satisfied by parabolic manifolds. We review several known parabolicity criteria and obtain some appropriate geometric conditions that guarantee the parabolicity of a submanifold described as a graph, and in particular of any slice of the spacetime (see Proposition~\ref{prop:main:2} and Corollary~\ref{prop:main:1}).
	
	The main results in the manuscript are contained in Sections~\ref{sec:tecnic} and~\ref{sec:trapped}.} {In Section~\ref{sec:tecnic} we obtain some interesting results regarding the geometry of spacelike submanifolds, of arbitrary dimension and codimension, immersed in an orthogonally splitted spacetime. In particular, we present in Theorems~\ref{thm:rig_general} and~\ref{thm:main:3} two different general rigidity theorems for spacelike parabolic submanifolds under some assumptions, including a boundedness hypothesis. In Theorem~\ref{thm:non_ex_min} we show that even if we impose neither parabolicity nor boundedness we can conclude the non-existence of a strict local minimum of our distinguished function. The results in this section are obtained under some technical assumptions, however such hypothesis are natural in the context of trapped surfaces. For that reason, in the last section of the manuscript we focus our attention on trapped surfaces immersed in realistic 4-dimensional spacetimes, obtaining some nice consequences. In particular, Theorem~\ref{thm:main:1} is derived from Theorem~\ref{thm:rig_general}, and it shows a non-existence result for parabolic (nearly, weakly, marginally) trapped surfaces and a rigidity result for parabolic extremal surfaces. Similarly, as a consequence of Theorem~\ref{thm:non_ex_min} we obtain Theorem~\ref{trappedMinMax}, which provides some information regarding the shape of (non-necessarily parabolic) trapped surfaces. Finally, some nice consequences of these theorems are obtained when restricted to some particular subfamilies of orthogonally splitted spacetimes that represent cosmological models of great interest, such as standard static, twisted and doubly twisted spacetimes, and in particular Generalized Robertson-Walker spacetimes.}

{Our results in Section~\ref{sec:trapped} generalize some previous results in the existing literature. Specifically, when applied to standard static spacetimes, we extend Proposition 2 in~\cite{MarsSeno} and Theorem 4 in~\cite{Pel}, and in the case of Generalized Robertson-Walker spacetimes we get in particular Results 5.1 and 5.2 in~\cite{IB} and most of the rigidity and non-existence results in Section 4.1 in~\cite{Pel2}.}

\section{On orthogonally splitted spacetimes}
\label{sec:orthogospace}

{Bernal and S\'anchez showed in~\cite{BS05} a relevant result on the topological and differentiable structure of any globally hyperbolic spacetime.} Specifically, they proved that any globally hyperbolic spacetime $(\overline{M}^{n+m+1},\overline{g})$, {$n\geq 2$, $m\geq 0$}, is isometric to a product $\mathbb{R}\times \mathcal{S}$ endowed with a Lorentzian metric \[\overline{g}=-\beta d\mathcal{T}^2+\hat{g},\] where $\mathcal{S}$ is a smooth spacelike Cauchy hypersurface of $\overline{M}$, $\mathcal{T}:\mathbb{R}\times\mathcal{S}\rightarrow \mathbb{R}$ the natural projection, $\beta\in\mathcal{C}^{\infty}(\mathbb{R}\times\mathcal{S})$ a positive smooth function and $\hat{g}$ a $(0,2)$ symmetric tensor field on $\mathbb{R}\times\mathcal{S}$ satisfying the following conditions:
\begin{itemize}
	\item[(a)] $\nabla \mathcal{T}$ is a past pointing timelike vector field on $\overline{M}$,
	\item[(b)] for every $\mathcal{T}_0\in\mathbb{R}$ the slice $\mathcal{S}_{\mathcal{T}_0}=\left\{\mathcal{T}_0\right\}\times  {\mathcal{S}}$ is a Cauchy hypersurface and $\hat{g}_{\mathcal{T}_0}:=\hat{g}_{|\mathcal{S}_{\mathcal{T}_0}}$ is a Riemannian metric on $\mathcal{S}_{\mathcal{T}_0}$, and
	\item[(c)] the radical of $\hat{g}$ at each point $q\in \mathbb{R}\times\mathcal{S}$, i.e. the subset of vectors that are orthogonal to the entire vector space, is given by $\mathrm{span}\left\{\nabla\mathcal{T}_q\right\}$.
\end{itemize}
Furthermore, the same authors proved in~\cite{BS06} that any Cauchy hypersurface of a globally hyperbolic spacetime $\overline{M}$ determines an orthogonal splitting as above.

{Motivated by the above splitting result,} and following the notation in~\cite{ARS19}, we will say that a spacetime $(\overline{M}^{n+m+1},\overline{g})$ is \emph{orthogonally splitted} if it is isometric to a product manifold
\begin{equation}\label{eq:1}
	\overline{M}^{n+m+1}=I\times F^{n+m},\quad \quad \overline{g}=-\beta\pi_{\mathbb{R}}^\ast(dt^2)+\pi_F^\ast(g_t),
\end{equation}
where $I\subset\mathbb{R}$ is an open interval, $F^{n+m}$ is an $(n+m)$-dimensional smooth manifold, $\pi_\mathbb{R}$ and $\pi_F$ denote the canonical projections onto $I$ and $F$ respectively, $\beta\in\mathcal{C}^\infty(\overline{M})$ is a positive smooth function and $\{g_t\}_{t\in I}$ is a {smooth} one-parametric family of Riemannian metrics on $F$. For simplicity, we will write the metric $\overline{g}$ as \[\overline{g}=-\beta dt^2+g_t.\]

Let us observe that since $\pi_{\mathbb{R}}=t$ is a temporal function, orthogonally splitted spacetimes are stably causal. However, we have no simple condition to ensure when an orthogonally splitted spacetime is globally hyperbolic. In fact, {even if for all $t\in I$ the completeness of $\left(F,\dfrac{1}{\beta}g_t\right)$ was guaranteed, we could only ensure that $t$ is a \emph{Cauchy temporal function},  {but not necessarily} that the spacetime is globally hyperbolic unless some additional hypothesis are considered (see \cite[Proposition 3.1]{sanchez2021}). Observe, however, that the above {completeness} condition is sufficient to guarantee global hyperbolicity for several families of orthogonally splitted spacetimes, including the standard static ones,} (see \cite[Section 3]{sanchez2021} for details in the topic).

\smallskip

It is natural to consider certain conditions in the expansive or contracting nature of a spacetime, and in particular of an orthogonally splitted spacetime. In this context, an orthogonally splitted spacetime $(\overline{M},\overline{g})$, defined as in~\eqref{eq:1}, is said to be \emph{expanding} (resp. \emph{non-contracting}) throughtout $\partial_t$ at a certain $t_0\in I$ if and only if
\begin{equation}\label{eq:expanding}
	\partial_t \beta_{|t_0} {\leq 0} \quad \hbox{and}\quad (\mathcal{L}_{\partial_t}g_t)_{|t_0}(v,v)>0,\;\;\hbox{(resp. $\geq$)},
\end{equation}
for any $v\in T_xF$, $x\in F$, {where $\mathcal{L}$ stands for the Lie derivative}. The spacetime is said to be \emph{expanding} (resp. \emph{non-contracting}) at an interval $I$ if and only if it is expanding (resp. non-contracting) at any $t\in I$. Analogously, $(\overline{M},\overline{g})$ is said to be \emph{contracting} (resp. \emph{non-expanding}) throughtout $\partial_t$ at $t_0\in I$ if and only if 
\begin{equation}\label{eq:contracting}
	\partial_t \beta_{|t_0} {\geq 0} \quad \hbox{and}\quad (\mathcal{L}_{\partial_t}g_t)_{|t_0}(v,v)<0,\;\;\hbox{(resp. $\leq$)},
\end{equation}
for any $v\in T_xF$, $x\in F$. Similarly, the spacetime is said to be \emph{contracting} (resp. \emph{non-expanding}) at an interval $I$ if and only if it is contracting (resp. non-expanding) at any $t\in I$. {Let us remark that both conditions for being non-contracting (resp. non-expanding) at a certain $t_0\in I$ are equivalent to ask $(\mathcal{L}_{\partial_t}\overline{g})_{|t_0}$ to be positive semi-definite (resp. negative semi-definite).}

Furthermore, the spacetime is said to be \emph{monotonic} if it is either non-contracting or non-expanding and \emph{strictly monotonic} if it is either expanding or contracting. Finally, the spacetime has a \emph{local contracting} (resp. \emph{expanding}) \emph{phase change} at $t_0$ if there exists $\epsilon>0$ such that for all $t_0-\epsilon {\,\,<\,\,} t < t_0$ the spacetime is expanding (resp. contracting), and for all $t_0< t {\,\,<\,\,} t_0+\epsilon$ the spacetime is contracting (resp. expanding) in $t$.

As it is explicitly stated in the previous definitions, the concepts of expansion and contraction depend on the timelike coordinate vector field $\partial_t$ relative to the given splitting. This terminology responds to a physical interpretation. On the one hand, let us observe that if we consider the observer field $U=\frac{1}{\sqrt{\beta}}\partial_t$, the proper time $\bar\tau$ of the observers in $U$ is given by \begin{equation}
	\label{eq:proper_time}
	d\bar\tau=\sqrt{\beta}dt.
\end{equation} 
Therefore, the condition $\partial_t\beta {\leq} 0$ (resp. $\partial_t\beta {\geq} 0$) means that the rate of change of the proper time with respect to the temporal function $t$ is non-increasing (resp. non-decreasing). Moreover, the Lorentzian length $|\partial_t|=-\sqrt{\beta}$ is non-decreasing (resp. non-increasing). On the other hand, the condition $(\mathcal{L}_{\partial_t}g_t)_{|t_0}(v,v)> 0$ for any $v\in T_xF$, $x\in F$, (resp. $(\mathcal{L}_{\partial_t}g_t)_{|t_0}(v,v)< 0)$) implies that an observer in $U$ perceives expansion (resp. contraction) in any direction of its rest space. Let us also observe from~\eqref{eq:proper_time} that $\sqrt{\beta}$ represents the derivative of the proper time with respect to the coordinate $t$.

In this context, we can define the following property of an orthogonally splitted spacetime. 

\begin{definition}\label{def:mod_prop_time}
	{Let $(\overline{M},\overline{g})$ be an orthogonally splitted spacetime as in~\eqref{eq:1}. Given a subset $\mathcal{S}\subseteq\overline{M}$, it is said that $\overline{M}$ has \emph{moderate proper time rate over $\mathcal{S}$} if $\mathrm{inf}_{\mathcal{S}}\,\beta >0$ and $\mathrm{sup}_{\mathcal{S}}\,\beta < \infty$. In the case where $\mathcal{S}=\overline{M}$, we will just say that $\overline{M}$ has \emph{moderate proper time rate}.}
\end{definition}

{The condition of having moderate proper time rate over a subset $\mathcal{S}$} translates into a control of the timelike component in the metric in~\eqref{eq:1} over the subset $\mathcal{S}$, and it is immediately satisfied in the case where $\beta$ is constant, and so for a large family of spacetimes such as twisted or Generalized Robertson-Walker spacetimes.

Let us finish this section by giving a closer look to this last case where $\beta$ is assumed to be constant, {that we can suppose without loss of generality to be $1$}. Our aim is {to show that the sectional curvature of timelike planes and the shape operator of a particular slice $F_{t_0}=\{t_0\}\times F$, $t_0\in\mathbb{R}$ determine the non-contracting/non-expanding behaviour of its chronological future/past. Let us recall at this point that the chronological future (resp. past) of a subset $S\subseteq \overline{M}$ is defined as the set of points in $\overline{M}$ that can be reached from $S$ by future directed (resp. past directed) timelike curves. As usual it will be denoted by $I^+(S)$ (resp. $I^-(S)$).} 

{Since} $\partial_t$ is unitary and normal to each slice $F_t$, the shape operator on such a slice, $A_t$, is defined by the Weingarten equation as ${A_t}(X)=\overline{\nabla}_X\partial_t$, {for all $X\in\mathfrak{X}(F_t)$}, {where $\overline{\nabla}$ denotes the Levi-Civita connection of $\overline{M}$}. { Moreover, $\partial_t$ is {obviously a closed vector field}, so it follows that}

{\[\frac{1}{2}\left( \mathcal{L}_{\partial_t}\overline{g} \right)({X},Y)=\overline{g}\left( \overline{\nabla}_{{X}}\partial_t,Y \right),\] for all ${X},Y\in \mathcal{X}(\overline{M})$}. 

\medskip

With these observations we are in conditions to prove the following result:

\begin{proposition}\label{auxiliarybeta1}
	Let $(\overline{M}^{n+m+1},\overline{g})$ be an orthogonally splitted spacetime with $\beta\equiv 1$. Let $F_{t_0}$ be a spacelike slice whose shape operator is {positive} (resp. negative) semi-definite. Assume also that the sectional curvature {of timelike planes} of the spacetime are non-negative. Then $(I^+(F_{t_0}),\overline{g})$ (resp. $(I^-(F_{t_0}),\overline{g})$) is a non-contracting (resp. non-expanding) orthogonally splitted spacetime.
\end{proposition}
\begin{proof}
	Let us consider the case where the shape operator over $F_{t_0}$ is {positive} semi-definite, as the other case will be analogous (with an appropriate time orientation change). For this, let us begin by taking $\gamma$ the {future directed} integral curve of $\partial_t$ starting at a point $\gamma(0)=p\in F_{t_0}$, and consider $X$ {any spacelike} vector field around $p$ which commutes with $\partial_t$. From the decomposition of an orthogonally splitted spacetime and the fact that $\beta\equiv 1$, we have that $\gamma$ is a future directed timelike geodesic, hence a standard computation leads us to:
	\begin{equation}
		\label{eq:32}
		\begin{split}
			-\overline{g}(\overline{R}(\partial_{t},X)\partial_t,X) &= \overline{g}(\overline{\nabla}_{\partial_t}\overline{\nabla}_X\partial_t,X) \\&=\dfrac{1}{2}\gamma' \left( \left( \mathcal{L}_{\partial_t}\overline{g} \right)(X,X) \right)-\overline{g}(\overline{\nabla}_X\partial_t,\overline{\nabla}_X\partial_t),
		\end{split}
	\end{equation}
	{where $\overline{R}(X,Y)Z=\overline{\nabla}_{[X,Y]}Z-[\overline{\nabla}_X,\overline{\nabla}_Y]Z$ stands for the Riemann curvature tensor of $\overline{M}$.}
	
	Therefore, recalling that $\overline{\nabla}_X\partial_t$ is spacelike and the hypothesis for sectional curvatures, we get that $\gamma' \left( \left( \mathcal{L}_{\partial_t}\overline{g} \right)(X,X)\right)\geq 0$. {In the particular case where $X\in\mathfrak{X}(F_{t_0})$ it holds} $\left( \mathcal{L}_{\partial_t}\overline{g} \right)(X,X)=2\overline{g}({A_{t_0}}(X),X)\geq 0$, {so it is} also non-negative along the curve $\gamma$. Finally, as any point in $I^+(F_{t_0})$ can be reached by a curve $\gamma$ as before {and, trivially, $(\mathcal{L}_{\partial_t}\overline{g})(\partial_t,\partial_t)$=0}, it follows that $\left( \mathcal{L}_{\partial_t}\overline{g} \right)(V,V)\geq 0$ for all points in $I^+(F_{t_0})$ and for all $V\in \mathfrak{X}(I^+(F_{t_0}))$, as desired. 
\end{proof}

\section{On submanifolds in orthogonally splitted spacetimes}
\label{sec:submanifolds}
Let us begin by recalling some general facts of submanifolds in Lorentzian Geometry. For this, consider $(\overline{M}^{n+m+1},\overline{g})$ a (general) spacetime, and let $x:\Sigma^n\rightarrow \overline{M}^{n+m+1}$, $n\geq 2$, {$m\geq 0$}, be a connected spacelike isometric immersion. Let us recall at this point that an isometric immersion is said to be \emph{spacelike} if the induced metric from $({\overline{M}},\overline{g})$, denoted by $g$, is a Riemannian one. It is worth pointing out that, from a physical point of view, the concept of submanifold usually corresponds to the mathematical concept of an embedded submanifold. However, our techniques will allow us to {present} our results in the more general framework of immersed submanifolds.

Let us denote by $\nablat$ and $\nabla$ the Levi-Civita connections of $\overline{M}$ and $\Sigma$, respectively. As it is well known, both Levi-Civita connections are related by the Gauss formula given by
\begin{equation}\label{eq:Gaussf}
	\nablat_XY=\nabla_XY-II(X,Y),
\end{equation}
for any $X,Y\in\mathfrak{X}(\Sigma)$, where $II:\mathfrak{X}(\Sigma)\times\mathfrak{X}(\Sigma)\rightarrow \mathfrak{X}^\perp (\Sigma)$ stands for the second fundamental form of $\Sigma$, given by
\begin{equation}\label{eq:2ff}
	II(X,Y)=-(\nablat_XY)^\perp,
\end{equation}	
$Z^\perp$ being the normal component of the vector field $Z\in\mathfrak{X}(\overline{M})$ along $\Sigma$. Observe that our convention in this paper for the definition of the second fundamental form is the classical in General Relativity, and opposite to the usual one in Differential Geometry. Also, as it is usual in the context of General Relativity, we define the mean curvature vector field of $\Sigma$ as the contraction of {its second fundamental form}, i.e.
\begin{equation}\label{eq:H}
	\overrightarrow{H}=\sum_{i=1}^n II(E_i,E_i),
\end{equation}
$\{E_1,...,E_n\}$ being a local orthonormal frame on $\Sigma$.

Furthermore, consider $S^{n+m}$ an immersed spacelike hypersurface of $(\overline{M}^{n+m+1},\overline{g})$ {and suppose that {$\Sigma^n$} is a Riemannian submanifold of $S$}. Let $\overrightarrow{h}$ denote the mean curvature vector of $\Sigma$ in $S$ and $A$ the shape operator of $S$ with respect to a unitary normal vector field $N$. Then, from the Gauss formula it is not difficult to see that
\begin{equation}\label{eq:H_desc}
	\overrightarrow{H}=\overrightarrow{h}+(\mathrm{tr}_{|\Sigma}A)N,
\end{equation}
$\mathrm{tr}_{|\Sigma}$ being the trace of the shape operator $A$ restricted to $\mathfrak{X}(\Sigma)$.

\medskip

From now on, let us assume that our spacetime $(\overline{M},\overline{g})$ is an orthogonally splitted spacetime as in~\eqref{eq:1}. For a given submanifold $\Sigma$ as before, it is possible to consider a distinguished function defined over $\Sigma$. Specifically, let us consider the function $\tau:=\pi_{\mathbb{R}{|\Sigma}}:\Sigma \rightarrow \mathbb{R}$. Observe that $\tau=t_0$ for a certain constant $t_0$ if and only if $\Sigma$ is contained in the spacelike slice $F_{t_0}$. Furthermore,  {the function $\tau$ allows us to define the following boundedness notions.} 

\begin{definition}\label{def:bounded}
	{Let $(\overline{M},\overline{g})$ be an orthogonally splitted spacetime as in~\eqref{eq:1}. Then, an immersed spacelike submanifold $\Sigma^n$ in $\overline{M}$ is said to be \emph{bounded away from the future infinity (resp. past infinity)} if $\tau$ is bounded from above (resp. from below). Finally, $\Sigma^n$ is said to be \emph{bounded {away} from the infinity} if it is both bounded {away} from the future and from the past infinity.}
\end{definition}

Our next goal is to compute a nice expression for the Laplacian of $\tau$ which will be essential for the main results of the paper. To that end, in what follows let $\left\{E_1,...,E_n\right\}$ and $\left\{N_1,...,N_{m+1}\right\}$ be local orthonormal frames of $\mathfrak{X}(\Sigma)$ and $\mathfrak{X}^\perp(\Sigma)$ respectively, and let us consider $\epsilon_j=g(N_j,N_j)$, $1\leq j\leq m+1$. It is quite straightforward to see that the gradient of $\tau$ in $\Sigma$, $\nabla \tau$, can be expressed as
\begin{equation}
	\label{eq:4}
	\nabla \tau =-\frac{1}{\beta} \partial_t^{T} = -\frac{1}{\beta}\left( \partial_t-\sum_{j=1}^{m+1} \epsilon_j \overline{g}(\partial_t,N_j)N_j \right),
\end{equation}
where $\partial_t^T$ denotes the tangential component of $\partial_t$ along $\Sigma$. With the previous expression, we can compute the Laplacian of $\tau$ in $\Sigma$, $\Delta \tau$, in the following way,
\begin{equation}
	\label{eq:5}
	\begin{aligned}
		\Delta \tau = \sum_{i=1}^n g\left(\nabla_{E_i}\nabla \tau, E_i\right) = \sum_{i=1}^n g\left(\nabla_{E_i}\left( -\frac{1}{\beta}\partial_t^T \right), E_i\right)\\
		= \sum_{i=1}^n \left( -E_i\left(\frac{1}{\beta}\right)g\left(\partial_t^{T}, E_i\right) -\frac{1}{\beta}g\left(\nabla_{E_i}\partial_t^{T},E_i\right) \right).
	\end{aligned}
\end{equation}
Let us recall that, on the one hand,
\begin{equation}
	\label{eq:6}
	\sum_{i=1}^n E_i\left(\frac{1}{\beta}\right)g\left(\partial^{T}_t,E_i\right) = -\frac{1}{\beta}g\left(\partial_t^T, \nabla \ln \beta\right)= g\left(\nabla \tau, \nabla \ln \beta\right).
\end{equation}
On the other hand,  
\begin{equation}
	\label{eq:7}
	g\left(\nabla_{E_i}\partial_t^T,E_i\right)= \overline{g}\left(\overline\nabla_{E_i}\partial_t,E_i\right) - \sum_{j=1}^{m+1} \epsilon_j  \overline{g}(\partial_t, N_j) \overline{g}(\overline{\nabla}_{E_i}N_j, E_i).
\end{equation}
Moreover,
\begin{equation}
	\label{eq:8}
	\overline{\mathrm{div}}(\partial_{t}) = \sum_{i=1}^{n} \overline{g}(\overline\nabla_{E_i}\partial_t,E_i) + \sum_{j=1}^{m+1} \epsilon_j \overline{g}(\overline\nabla_{N_j}\partial_t,N_{j}),
\end{equation}
where $\overline{\mathrm{div}}$ stands for the divergence operator on $\overline{M}$. There are some recognizable terms there. From~\eqref{eq:2ff} and~\eqref{eq:H}, it follows that
\begin{equation}
	\label{eq:9}
	\overline{g}(\partial_t,\overrightarrow{H}) =   \sum_{j=1}^{m+1}\epsilon_j \overline{g}\left(\partial_t,N_j\right)\sum_{i=1}^n\overline{g}\left(\overline\nabla_{E_i}N_j, E_i\right).
\end{equation}
For the rest of the terms, we should make use of the following decomposition of the normal vector fields, $N_j=-\frac{1}{\beta} g(\partial_t,N_j)\partial_t+N^F_j$, where $N^{F}_j\in \mathfrak{X}(F)$ and it satisfies that $[\partial_t, N^F_{j}]=0$. Then,
\begin{equation}
	\label{eq:11}
	\begin{aligned}
		\overline{g}(\overline\nabla_{N_j}\partial_t,N_j)= \frac{1}{\beta^2} \overline{g}(\partial_t,N_j)^2 \overline{g}(\overline\nabla_{\partial_t}\partial_t,\partial_t) + \overline{g}(\overline\nabla_{N^F_j}\partial_t, N^F_j)\qquad\\ -\frac{1}{\beta}\overline{g}(\partial_t,N_j)\left( \overline{g}(\overline\nabla_{N^F_j}\partial_{t}, \partial_t)- \overline{g}(\overline\nabla_{\partial_t} N^F_j, \partial_{t}) \right)\quad\\
		= -\frac{1}{2} \overline{g}\left(\frac{\partial_t}{\sqrt{\beta}},N_j\right)^2 \partial_t(\ln \beta) +\frac{1}{2} (\mathcal{L}_{\partial_t} \overline{g})(N_j^{F}, N^F_j).
	\end{aligned}
\end{equation}

Joining {the above expressions} with (\ref{eq:5}) we get

\begin{equation}
	\label{eq:13}
	\begin{aligned}
		\Delta \tau =& -g\left(\nabla \tau, \nabla \left(\ln \beta\right)\right) -\frac{1}{\beta}\left(\overline{\mathrm{div}}(\partial_t) - \overline{g}(\partial_t,\overrightarrow{H})\right)\\ &+\dfrac{1}{\beta}\left(\sum_{j=1}^{m+1}\epsilon_j \frac{1}{2}\left( (\mathcal{L}_{\partial_t}g_t)(N_j^F,N_j^F) - \overline{g}\left(\frac{\partial_t}{\sqrt{\beta}},N_j\right)^2\partial_t(\ln \beta) \right)  \right).
	\end{aligned}
\end{equation}

Finally, and in order to give a more compact expression, let us make the following two observations. On the one hand, let us observe that in~\eqref{eq:8} we have computed $\overline{\textrm{div}}(\partial_t)$ considering the local orthonormal frame of $\mathfrak{X}(\overline{M})$ given by $\mathcal{B}=\left\{E_1,...,E_n,N_1,...,N_{m+1}\right\}$. However, it can also be useful to consider a different orthonormal frame of $\mathfrak{X}(\overline{M})$ related to the splitting in (\ref{eq:1}). In this sense, let us define $\mathcal{B}'=\left\{\partial_t/\sqrt{\beta}, \tilde{E}_1,...,  \tilde{E}_{m+n}\right\}$, where given any point $p=(t,q)\in \overline{M}$, $\{\tilde{E}_{1\,|p},...,\tilde{E}_{m+n\,|p}\}$ is the lifting to $T_p\overline{M}$ of an orthonormal basis of $T_qF$ with respect to the Riemannian metric $g_t$. Thus, computing $\overline{\textrm{div}}(\partial_t)$ with respect to $\mathcal{B}'$, we get
\begin{equation}
	\label{eq:12}
	\overline{\mathrm{div}}(\partial_t) = \frac{1}{2}\left(\partial_t (\ln \beta) + \sum_{i=1}^{n+m} (\mathcal{L}_{\partial_t}g_t)\left( \tilde{E}_i, \tilde{E}_i \right)  \right).
\end{equation}

On the other hand, let us define the map $\overline{\xi}:{\mathfrak{X}(\overline{M})}\times {\mathfrak{X}(\overline{M})} \rightarrow \mathbb{R}$ given by
\begin{equation}
	\label{eq:14}
	\overline{\xi}(V,W) = (\mathcal{L}_{\partial_t}g_t)(d\pi_F(V),d\pi_F(W)).
\end{equation}
It follows easily that the trace of $\overline{\xi}$ can be computed by using any of the bases we have considered, obtaining that
\begin{equation}
	\label{eq:15}
	\begin{aligned}
		\mathrm{tr}(\overline{\xi}) =& \sum_{i=1}^n (\mathcal{L}_{\partial_t}g_t)(d\pi_F(E_i),d\pi_F(E_i)) + \sum_{j=1}^{m+1}\epsilon_j(\mathcal{L}_{\partial_t}g_t)(d\pi_F(N_j),d\pi_F(N_{j}))\\
		=&\;  \mathrm{tr}(\xi) + \sum_{j=1}^{m+1}\epsilon_j(\mathcal{L}_{\partial_t}g_t)(N^{F}_j,N^{F}_{j}),
	\end{aligned}
\end{equation}
where $\xi=\overline{\xi}_{|{\mathfrak{X}(\Sigma)}\times {\mathfrak{X}(\Sigma)}}$, and
\begin{equation}\label{eq:divB'}
	\mathrm{tr}(\overline{\xi})=\sum_{i=1}^{n+m} (\mathcal{L}_{\partial_t}g_t)\left( \tilde{E}_i, \tilde{E}_i \right),
\end{equation}
so in particular
\begin{equation}
	\label{eq:16}
	\overline{\mathrm{div}}(\partial_t) = \frac{1}{2}\left( \partial_t(\ln \beta) + \mathrm{tr}(\overline{\xi}) \right).
\end{equation}

In conclusion, taking into account (\ref{eq:15}) and (\ref{eq:16}), (\ref{eq:13}) becomes

\begin{equation}
	\label{eq:17}
	\begin{aligned}
		\Delta \tau = & -g(\nabla \tau, \nabla (\ln \beta)) -\frac{1}{\beta}\left( -\overline{g}(\partial_t, \overrightarrow{H})+\frac{1}{2}\mathrm{tr}(\xi) + \frac{1}{2}\partial_t(\ln \beta)\left( 1 + \sum_{j=1}^{m+1}\epsilon_j\overline{g}\left(\frac{\partial_t}{\sqrt{\beta}},N_j\right)^2 \right)\right). 
	\end{aligned}
\end{equation}
Moreover, if we recall that $\overline{g}(\partial_t/\sqrt{\beta}, \partial_t/\sqrt{\beta})=-1$, then

\[
-1=\sum_{i=1}^{n}\overline{g}\left(\frac{\partial_t}{\sqrt{\beta}},E_i  \right)^2 + \sum_{j=1}^{m+1}\epsilon_j\overline{g}\left(\frac{\partial_t}{\sqrt{\beta}},N_j  \right)^2=\sinh^2(\theta)-\cosh^2(\theta),
\] for certain $\theta$. Hence, the last term in the parenthesis is easily identifiable with $-\sinh^2(\theta)$, and therefore,
\begin{equation}
	\label{eq:18}
	\Delta \tau = -g(\nabla\tau, \nabla(\ln \beta)) -\frac{1}{\beta}\left(-\overline{g}(\partial_t, \overrightarrow{H}) + \frac{1}{2} \left( \mathrm{tr}(\xi) - \partial_t(\ln \beta)\sinh^2(\theta) \right)  \right).
\end{equation}

As a final step, let us consider a conformal change on the metric $g$, so the first term on the previous expression can be absorved into the Laplacian. Concretely, and taking into account that under the conformal change $\tilde{g}=e^{2\varphi}g$ the Laplacian transforms as
\begin{equation}
	\label{eq:19}
	\tilde{\Delta} f = e^{-2\varphi}\left( \Delta f + (n-2)g(\nabla f, \nabla \varphi) \right),
\end{equation}
$\tilde{\Delta}$ being the Laplacian with respect to $\tilde{g}$, if we assume that $n>2$ and we consider $\varphi = \frac{1}{n-2}\ln \beta$, we finally get
\begin{equation}
	\label{eq:20}
	\tilde{\Delta} \tau=-\dfrac{1}{2} \beta^{\frac{n}{n-2}} \left(\mathrm{tr}(\xi) - \partial_t(\ln \beta)\sinh^2(\theta)-2\overline{g}(\partial_t,\overrightarrow{H}) \right).
\end{equation}
As we have mentioned before, the previous expression for the Laplacian will be essential in order to obtain the main results of this paper. Observe that, from the computations, \eqref{eq:20} is only valid under the assumption $n>2$. However, we will see that such a restriction can be removed from our main results.

\section{Parabolicity of spacelike submanifolds}
\label{sec:parabolicity}

Let us recall that a smooth function $f\in\mathcal{C}^\infty(M)$ over a Riemannian manifold $(M,g)$ is said to be superharmonic (resp. subharmonic) if its Laplacian is non-positive (resp. non-negative). Then, a complete (non-compact) $n$-dimensional Riemannian manifold is called \emph{parabolic} if it does not admit any non-constant positive superharmonic function on it (see, for instance,~\cite{Kaz}).

The study of parabolicity can be approached from different points of view. For instance, parabolicity of Riemannian surfaces is closely related to their Gaussian curvature. On the one hand, a key result by Ahlfors~\cite{Ahl} and Blanc-Fiala-Huber~\cite{Hub} asserts that a complete (non-compact) Riemannian surface with non-negative Gaussian curvature is parabolic. Moreover, let us recall that the Bonnet-Myers' theorem (see \cite{Petersen}) ensures that any Riemannian surface with Gaussian curvature bounded from below by a positive constant has to be compact. Hence, the condition of non-negativity for the Gaussian curvature becomes naturally an extension (from the geometric perspective) of the compact case.

{In a more general setting}, it is also well-known that any complete Riemannian surface with \emph{finite total curvature} is also parabolic (see~\cite{Lee}). Let us recall that a complete Riemannian surface  {$M$} is said  to have finite total curvature if the negative part of its Gaussian curvature is integrable. More precisely, if $K$ denotes the Gaussian curvature on {$M$}, then {$M$} has finite total curvature if 
\begin{equation}\label{Parabolic}
	\int_{{M}} \max (0,- K)\, dA_{{M}}<\infty,
\end{equation}
where $dA_{{M}}$ is the area element of {$M$}. Observe that the integral is well-defined with a compact exhaustion procedure. As a direct consequence, any complete Riemannian surface with non-negative Gaussian curvature out of a compact subset is parabolic.

\smallskip

In spite of the $2$-dimensional case, in higher dimension parabolicity of Riemannian manifolds has no clear relation with the sectional curvature. In fact, the Euclidean space $\mathbb{R}^n$ is parabolic if and only if $n \leq 2$. However, there are criteria assuring the parabolicity of a Riemannian manifold of arbitrary dimension based on the volume growth of its geodesic balls (see \cite{AMR} and references therein).

Nevertheless, for this paper we will make use of a different approach. As it is obvious from its definition, and we have mentioned in the Introduction, parabolicity can be interpreted as a kind of extension of the notion of compactness. In fact, some properties {of the geometric analysis on}  compact Riemannian manifolds are also present in parabolic ones. However, and unlike the compact case, the image by a differentiable map of a parabolic manifold is no longer parabolic. In order to preserve parabolicity, we need to restrict ourselves to the so-called quasi-isometries. Let us recall that given two Riemannian manifolds $(M_1, g_1)$ and $(M_2, g_2)$, a global diffeomorphism $\varphi:M_1\rightarrow M_2$ is called a quasi-isometry if there exists a constant $c \geq 1$ such that
\begin{equation}\label{eq:q-i}
	c^{-1} \ g_1(v, v) \leq g_2(d\varphi(v), d\varphi(v)) \leq c \ g_1(v, v),
\end{equation}
for all $v \in T_pM_1$, $p \in M_1$. It is known that parabolicity is invariant under quasi-isometries \cite{Ka}, \cite[Cor. 5.3]{Gr}.

\smallskip

{The concept of quasi-isometry will help us} to determine whether a submanifold is parabolic by making an appropriate comparison with a parabolic manifold. In this sense, and moving now to the context of Lorentzian geometry, let us consider $(M_1,\overline{g}_1)$ and $(M_2,\overline{g}_2)$ two Lorentzian manifolds. We will say that $(M_{1},\overline{g}_1)$ is \emph{causally related} to $(M_2,\overline{g}_2)$, and we will denote it by $M_1\prec M_2$, if there exists a global diffeomorphism $\phi:M_1\rightarrow M_2$ such that {for any $p\in M_1$ and} any future directed timelike vector $v\in T_pM_1$, {$d\phi_p(v)$} is also a future directed timelike vector in $T_{\phi(p)}M_2$. {From a physical point of view, $(M_1,\overline{g}_1)$ is causally related to $(M_2,\overline{g}_2)$ if the timelike cones for $(M_1,\overline{g}_1)$ are narrower than those for $(M_2,\overline{g}_2)$. Finally,} two Lorentzian manifolds $M_1$ and $M_2$ are called \emph{isocausal} if $M_1\prec M_2$ and $M_2\prec M_1$ (see~\cite{GP-Sen}).

Inspired by \cite{FHS1}, let us deal with the particular case of orthogonally splitted spacetimes $(\overline{M}=I\times F\,,\,\overline{g})$ as in~\eqref{eq:1}, which are causally related (in both directions) to Generalized Robertson-Walker spacetimes. Concretely, we will assume that there exists $g_0$ a Riemannian metric defined on $F$, and two positive functions $\alpha_i:\mathbb{R}\rightarrow \mathbb{R}^+$, $i=1,2$, such that for any point $p=(t,x)\in {\overline{M}}$ and $v\in T_xF$ it holds

\begin{equation}
	\label{eq:21}
	\beta(p)\alpha_2(t)g_0(v,v) \leq  g_t(v,v) \leq \beta(p)\alpha_1(t)g_0(v,v).
\end{equation}
{In that case, }it is not difficult to show (see for instance \cite[Section 4.2]{FHS1}) that $({\overline{M}},\overline{g}_1)\prec ({\overline{M}},\overline{g})\prec ({\overline{M}},\overline{g}_2)$, where
\begin{equation}
	\label{eq:22}
	\overline{g}_i=-dt^2 + \alpha_i(t) g_0,\quad 1\leq i\leq 2.
\end{equation}
Moreover, it follows that both $(\overline{M},\overline{g})$ and $(\overline{M},\overline{g}_1)$ {(or $(\overline{M},\overline{g}_2)$)} are isocausal if both metrics $\overline{g}_i$ are conformal, which follows if (see \cite[Proposition 4.6]{FHS1})
\begin{equation}
	\label{eq:23}
	\int\limits_0^{\pm\infty}\frac{1}{\sqrt{\alpha_1(t)}}dt = \int\limits_0^{\pm\infty}\frac{1}{\sqrt{\alpha_2(t)}}dt.
\end{equation}

\smallskip

Our aim now is to obtain some appropriate geometric conditions which guarantee the parabolicity of a submanifold $\Sigma$ of arbitrary codimension in an orthogonally splitted spacetime. In order to do this, we will restrict ourselves to the case of submanifolds obtained as a graph of a suitable function. Specifically, let us assume that $F$ can be written as a product manifold $F^{n+m}=F_1^{{n}}\times F_2^{{m}}$. Then, given any smooth function $u:F_1\rightarrow \mathbb{R}$ and any (fixed) point $x_2\in F_2$, we can define the submanifold $\Sigma_{u,x_2}^{{n}}$ as the graph
\begin{equation}
	\label{eq:24}
	\Sigma_{u,x_2} := \left\{ (u(x_1),x_1,x_2)\in {\overline{M}}: x_1\in F_1 \right\}.
\end{equation}

We are now able to present the following result.

\begin{proposition}
	\label{prop:main:2}
	Consider $({\overline{M}^{n+m+1}}=I\times F^{n+m},\overline{g})$ an orthogonally splitted spacetime as in \eqref{eq:1} with $F^{n+m}=F_1^n\times F_2^m$, $u:F_1\rightarrow \mathbb{R}$ a smooth function, and assume the following conditions:
	\begin{enumerate}[label=(\roman*)]
		\item \label{prop:main:2-1} there exist $g_0$ and $\alpha_i$, $1\leq i\leq 2$, satisfying~\eqref{eq:21},
		\item \label{prop:main:2-2} $(F_1,g_{0\,|F_1})$ is a parabolic Riemannian manifold,
		\item \label{prop:main:2-4}  $0<\inf\beta_u\alpha_2(u)$ and $\sup\beta_u\alpha_1(u)<\infty$, where $\beta_u=\beta(u(x_{{1}}),x_1,x_2)$ and $\alpha_i(u)=\alpha_i(u(x_{{1}}))$; and
		\item \label{prop:main:2-3} $\left| \nabla^0 u \right|_0\leq {k\,} \alpha_2(u)$ for some constant $0<{k}<1$, where $\nabla^0$ and $|\cdot|_0$ stand for the gradient operator and the norm in $(F_{{1}},g_{0\,|F_{{1}}})$.
	\end{enumerate}
	Then, for each $x_{{2}}\in F_{{2}}$ the submanifold $(\Sigma_{{u,x_2}},g)$ is parabolic.
\end{proposition}

\begin{proof}
	Given $x_2\in F_2$, our aim is to show that $(\Sigma_{u,x_2},g)$ and $(F_1,g_{0\,|F_1})$ are quasi-isometric. In order to do so, let us recall that $(\Sigma_{u,x_2},g)$ is isometric to $(F_1,g_{F_1})$, where
	\begin{equation}
		\label{eq:25}
		g_{F_1}(v,w)= g_{u(x_1)}(v,w) -{\beta_u} du(v)du(w),
	\end{equation}
	for any $v,w\in T_{x_1}F_1$. Then, it follows from the second inequality of~\eqref{eq:21} that, for any $v\in T_{x_1}F_1$,
	\begin{equation}
		\label{eq:26}
		g_{F_{1}}(v,v) \leq g_{u(x_1)}(v,v)\leq \beta_{u}\alpha_1(u) g_0(v,v).
	\end{equation}
	Furthermore, from the first inequality of~\eqref{eq:21} and assumption~\ref{prop:main:2-3} it holds
	\begin{equation}
		\label{eq:27}
		\begin{aligned}
			g_{F_{1}}(v,v) = &   g_{u(x_1)}(v,v) - \beta_{u}du(v)^2\\
			\geq& \beta_{u} \alpha_2(u) g_0(v,v) - \beta_{u} du(v)^2\\  \geq  & \beta_{u}\alpha_2(u)g_0(v,v)\left( 1-{k} \right).
		\end{aligned}
	\end{equation} Hence, joining together these last two inequalities,
	\begin{equation}
		\label{eq:31}
		(1-{k})\beta_u\alpha_2(u)g_0(v,v)\leq g_{F_1}(v,v)\leq \beta_u \alpha_1(u)g_0(v,v).
	\end{equation}
	In conclusion, taking into account~\eqref{eq:31} and assumption~\ref{prop:main:2-4}, both $g_{F_{{1}}}$ and $g_0$ are quasi-isometric. Finally, since $(F_1,g_0|_{F_{1}})$ is parabolic by assumption~\ref{prop:main:2-2}, $(F_1,g_{F_1})$ is also parabolic, and so {it} is $(\Sigma_{u,x_2},g)$.
\end{proof}

\begin{remark}
	\label{rem:main:1}
	Let us observe that, in the proof of Proposition~\ref{prop:main:2}, we really do not need that (\Ref{eq:21}) holds for all vectors in $T_{{p}}\overline{M}$, but only for those in $T_{{x_1}} F_1$. This will be specially relevant in Section \ref{sec:aplicaCos} when considering twisted products.
\end{remark}

The conditions imposed in the previous result  apply in different, and quite remarkable, situations. For instance, let us assume {that $F=F_1$, so that $u:F\rightarrow\mathbb{R}$ and \[\Sigma=\Sigma_{u}=\left\{(u(x_1),x_1)\in\overline{M}:x_1\in F\right\}\]}is a hypersurface. Observe that the hypothesis \textit{(i)} in Proposition \ref{prop:main:2} ensures that $(\overline{M},\overline{g})\prec (\overline{M},\overline{g}_2)$, hence if $(\Sigma,g)$ is spacelike, then it should also be spacelike $(\Sigma,\overline{g}_{2\,|\Sigma})$. A simple computation shows that $(\Sigma,\overline{g}_{2\,|\Sigma})$ is spacelike if and only if
\[
\left| \nabla^0 u \right|_0 < \alpha_2(u(x_1)),
\] for all $x_1\in F$. As the cones for $g$ are narrower than those in $\overline{g}_{2}$, the inequality $\left| \nabla^0 u \right|_0 \leq {k\,} \alpha_2(u({x_1}))$ for some ${k\,}<1$ follows, for instance, if we assume that the cones in $g$ are \emph{strictly contained} on those in $\overline{g}_{2\,|\Sigma}$, not being close to each other even asymptotically. 

With respect to the boundedness of the functions $\beta_u\alpha_1$ and $\beta_u\alpha_2$, let us analyse each function $\beta$ and $\alpha_i$ independently. Focusing on the function $\beta$, it is quite regular to ask the spacetime to have moderate proper time rate, or at least some kind of controlled asymptotic behaviour in the spatial directions. Concretely, for a fixed $t_0$, the function $\beta_{t_0}:F\rightarrow \mathbb{R}$ with $\beta_{t_0}(x)=\beta(t_0,x)$ is tipically assumed to have some kind of boundedness out of a compact set. For instance, by imposing convergence along \emph{radial} directions in spherically symmetric spacetimes. That is usually the case for asymptotically flat spacetimes, where the function $\beta_{t_0}$ should naturally converge to a non-zero value in asymptotic directions. In such a case, for each $t_0$ the function $\beta_{t_0}$ is bounded.

Regarding the functions $\alpha_i$, their boundedness is easily achievable if we restrict ourselves to submanifolds bounded away from the infinity, that is, if we consider only submanifolds living in a \emph{slab} of the spacetime, {i.e. a region between two slices}. We can also ensure that condition if we assume that the limit of $\alpha_i$ exists at the endpoints of the interval $I$. 

Finally, and as an immediate consequence of Proposition~\ref{prop:main:2}, we can give conditions for the parabolicity of any slice.

\begin{corollary}
	\label{prop:main:1}
	Consider $(\overline{M}^{n+m+1},\overline{g})$ an orthogonally splitted spacetime as in~\eqref{eq:1}. Assume that there exist functions $\alpha_i$ for $i=1,2$ and a Riemannian metric $g_0$ on $F$ satisfying (\Ref{eq:21}) and such that $(F,g_0)$ is a parabolic Riemannian metric. Then, for each $t\in \mathbb{R}$ where $\beta_t:F\rightarrow \mathbb{R}$ with $\beta_t(x)=\beta(t,x)$ is bounded, the slice $({F_t},g_t)$ is also a parabolic Riemannian manifold.
\end{corollary} 

\section{On the geometry of spacelike submanifolds in an orthogonally splitted spacetime}
\label{sec:tecnic}

The main goal of this section is to present some results regarding the geometry of {parabolic} spacelike submanifolds $\Sigma^n$ in an orthogonally splitted spacetime $(\overline{M},\overline{g})$ as in~\eqref{eq:1}. In fact, we will show that the geometry of such manifolds will be specially restricted inside orthogonally splitted spacetimes. There are several points of interest for the study of parabolic manifolds. On the one hand, parabolic manifolds are well known objects in mathematical physics since they appear, for example, in the description of the Brownian motion (see \cite{Gr}). On the other hand, as it was mentioned in Section \ref{sec:parabolicity}, parabolicity can be ensured by imposing a \emph{moderate volume growth} for geodesics balls (see \cite[Section 3]{Kaz}). Hence a spacelike submanifold will be parabolic if the expansion along (spacelike) radial geodesics has some controlled behaviour. Finally, in the two dimensional case, we show that any spacelike surface with non-negative Gaussian curvature also has such a restricted geometry.

In order to obtain the main results for parabolic spacelike submanifolds, we will impose some geometrical assumptions on the spacetime and we will ask the function on $\Sigma^n$ given by $\overline{g}(\partial_t,\overrightarrow{H})$ to be signed. Although this condition can seem a bit technical, in the case where $\overrightarrow{H}$ is {causal}, it is closely related to the concept of trapped submanifolds. In fact, very nice consequences will be obtained in this context. 

Before stating the main results of this section, recall that given any $p\in \overline{M}$, a vector $v\in T_p\overline{M}$ is said to be causal if and only if it is either timelike or lightlike, i.e. $\overline{g}(v,v)\leq 0$. As it is usual in General Relativity, we will consider $v=0$ as a lightlike vector field.

\smallskip

For most of the forthcoming results, the expression obtained in~\eqref{eq:20} will be essential. Let us recall that such an expression depends on the function $\beta$, a certain function $\theta$ and the trace of the map $\xi(V,W)=\left( \mathcal{L}_{\partial_t}g_t \right)(d\pi_{F}(V),d\pi_{F}(W))$, where $V,W\in \mathcal{X}(\Sigma)$. Our first result is a {general} rigidity theorem.

\begin{theorem}
	\label{thm:rig_general}
	Let $(\overline{M}^{n+m+1}=I\times F^{n+m},\overline{g})$ be an orthogonally splitted spacetime as in~\eqref{eq:1} and let $\Sigma^n$ be a spacelike parabolic submanifold in $\overline{M}$. Let us assume that the following conditions hold:
	\begin{enumerate}[label=(\roman*)]
		\item \label{thm:rig_general-1} $\mathrm{tr}\,\xi\geq 0$ and $\partial_{t}(\beta_{|\Sigma})\leq 0$ (resp. $\mathrm{tr}\,\xi\leq 0$ and $\partial_t(\beta_{|\Sigma})\geq 0$),
		\item \label{thm:rig_general-2} {$\overline{M}$ has moderate proper time rate over $\Sigma$},
		\item \label{thm:rig_general-3} the mean curvature vector of $\Sigma$ satisfies $\overline{g}(\partial_t,\overrightarrow H)\leq 0$ (resp. $\overline{g}(\partial_t,\overrightarrow H)\geq 0$), and
		\item \label{thm:rig_general-4} $\Sigma$ is bounded away from the past infinity (resp. from the future infinity).
	\end{enumerate}
	Then, $\Sigma$ is contained in a spacelike slice $F_{t_0}=\{t_0\}\times F$. Moreover, $\partial_t(\beta_{|\Sigma})=\mathrm{tr}\,\xi={\overline{g}(\partial_t,\overrightarrow{H})}=0$. {In particular, either $\overrightarrow{H}$ is spacelike or $\Sigma$ is minimal in $F_{t_0}$.}
\end{theorem}
\begin{proof}
	Let us assume firstly that $n>2$ and let us consider the conformal change of metric on $\Sigma$ given by $\tilde{g}=\beta^{\frac{2}{n-2}}g$. From assumption~\ref{thm:rig_general-2}, we immediately get that $(\Sigma, g)$ and $(\Sigma,\tilde{g})$ are quasi-isometric. Thus, as stated in Section~\ref{sec:parabolicity}, $(\Sigma,\tilde{g})$ is also parabolic.
	
	Taking into account assumptions~\ref{thm:rig_general-1} and~\ref{thm:rig_general-3} of the theorem,~\eqref{eq:20} yields $\tilde{\Delta}\tau\leq 0$. From~\ref{thm:rig_general-4} $\tau$ is bounded from below so, by the parabolicity of $(\Sigma,\tilde{g})$, $\tau=t_0$ for a certain $t_0\in\mathbb{R}$, i.e. $\Sigma$ is contained in the spacelike slice $F_{t_0}$. In particular, $\tilde{\Delta}\tau=0$, so from~\eqref{eq:20} we get $\partial_t(\beta_{|\Sigma})=\mathrm{tr}\,\xi=g(\partial_t,\overrightarrow{H})=0$. Consequently, either $\overrightarrow{H}$ is spacelike or it vanishes. The last assertion follows immediately from~\eqref{eq:H_desc}.
	
	In the case $n=2$, the conformal change cannot be applied. However, we can just consider the product manifolds $\overline{M}_\ast^{n+m+2}=\overline{M}^{n+m+1}\times\mathbb{S}^1$ and $\Sigma_\ast^3=\Sigma^2\times\mathbb{S}^1$ and apply the result to these new manifolds. In fact, just observe that such manifolds satisfy the assumptions of the theorem. Firstly, $\overline{M}_\ast=I\times (F^{n+m}\times \mathbb{S}^1)$ is also an orthogonally splitted spacetime with the metric $\overline{g}_\ast=-\beta dt^2+(g_t+g_{\mathbb{S}^1})$. Since the function $\beta$ remains the same and $g_{\mathbb{S}^1}$ does not depend on $t$, assumptions~\ref{thm:rig_general-1} and~\ref{thm:rig_general-2} also hold in this situation. Furthermore, since the product of a parabolic and a compact submanifold is also parabolic (see \cite{Kaz}), $\Sigma_\ast$ inherits the spatiality, parabolicity and boundedness conditions of $\Sigma$, so in particular~\ref{thm:rig_general-4} holds. Finally, it is easy to check that $\overrightarrow{H}_\ast=\overrightarrow{H}$, where $\overrightarrow{H}_\ast$ is the mean curvature vector of $\Sigma_\ast$ in $\overline{M}_\ast$, so~\ref{thm:rig_general-3} is also satisfied.
\end{proof}

\begin{remark}
	As direct consequences of Theorem~\ref{thm:rig_general} we can obtain some non-existence results. In fact, we only have to ask one of the inequalities in~\ref{thm:rig_general-1} or~\ref{thm:rig_general-3} to be strict in order to conclude that there do not exist any submanifold satisfying those assumptions. We will explicitly state such results in the context of trapped surfaces in Section~\ref{sec:trapped}.
\end{remark}

Even when $\Sigma$ is not a parabolic surface and we do not impose any boundedness assumption on $\beta$, we can also provide some information regarding the shape of spacelike submanifolds with signed $g(\partial_t,\overrightarrow{H})$.

\begin{theorem}
	\label{thm:non_ex_min}
	Let $(\overline{M}^{n+m+1}=I\times F^{n+m},\overline{g})$ be an orthogonally splitted spacetime as in~\eqref{eq:1} and let $\Sigma^n$ be a spacelike submanifold in $\overline{M}$. Let us assume that the following conditions hold:
	\begin{enumerate}[label=(\roman*)]
		\item \label{thm:non_ex_min-1} $\mathrm{tr}\,\xi\geq 0$ and $\partial_{t}(\beta_{|\Sigma})\leq 0$ (resp. $\mathrm{tr}\,\xi\leq 0$ and $\partial_t(\beta_{|\Sigma})\geq 0$),
		\item \label{thm:non-ex-min-3} the mean curvature vector of $\Sigma$ satisfies $\overline{g}(\partial_t,\overrightarrow H)\leq 0$ (resp. $\overline{g}(\partial_t,\overrightarrow H)\geq 0$).
	\end{enumerate} Then, $\tau$ cannot achieve a strict local minimum (resp. maximum) on $\Sigma$.
\end{theorem}
\begin{proof}
	Suppose that $n>2$. Let us assume by contradiction that $\tau$ attains a strict local minimum at a certain $p_0\in\Sigma$, and let us denote it by $t_0=\tau(p_0)$. Then, there exists a small enough compact neighbourhood $U$ of $p_{0}$ in $\Sigma$ such that $\tau_{|\partial U}=t_1$ is constant and $\tau-t_1\leq 0$ in $U$.
	
	Let us consider on $U$ the vector field $X=(\tau-t_1)\tilde{\nabla} \tau$. Since $X_{|\partial U}=0$, the divergence theorem yields
	
	\begin{equation}
		\label{eq:29}
		0=\int_U\mathrm{div}(X) = \int_U \left( \tilde{g}(\tilde{\nabla} \tau, \tilde{\nabla} \tau) + (\tau-t_1)\tilde{\Delta} \tau \right).
	\end{equation}
	{Following the reasoning} in Theorem~\ref{thm:rig_general}, under the assumptions of the theorem $\tilde{\Delta} \tau\leq 0$, and so, restricted to $U$, {it holds} $\left( \tau-t_1 \right)\tilde{\Delta} \tau\geq 0$. Hence, from~\eqref{eq:29} if follows that $\tilde{g}( \tilde{\nabla}\tau,\tilde{\nabla}\tau )= 0$. Consequently, $\tau$ is constant on $U$, which contradicts the fact of $p_0$ being a strict local minimum of $\tau$.
	
	In the case $n=2$ we can proceed as in the proof of Theorem~\ref{thm:rig_general}.
\end{proof}

\begin{remark}
	The {above} results have been stated under general technical assumptions. However, those assumptions are reasonable from a physical point of view. In fact, in Theorems~\ref{thm:rig_general} and~\ref{thm:non_ex_min} assumption~\ref{thm:rig_general-1} is weaker that asking the spacetime to be non-contracting (resp. non-expanding). Furthermore, {obviously} assumption~\ref{thm:rig_general-2} in Theorem~\ref{thm:rig_general} 	is weaker than asking $\overline{M}$ to have moderate proper time rate. Finally, as stated before, assumption~\ref{thm:rig_general-3} in Theorem~\ref{thm:rig_general} and assumption~\ref{thm:non-ex-min-3} in Theorem~\ref{thm:non_ex_min} are, as we will see in the forthcoming section, closely related to the concept of trapped submanifolds.
\end{remark}

We can also consider the case where the spacetime has a local phase change.

\begin{theorem}
	\label{thm:main:3}
	Let $(\overline{M}^{n+m+1}=I\times F^{n+m},\overline{g})$ be an orthogonally splitted spacetime as in~\eqref{eq:1} and let $\Sigma^n$ be a spacelike parabolic submanifold in $\overline{M}$. Let us assume that the following conditions hold: 
	
	\begin{enumerate}[label=(\roman*)]
		\item \label{thm:main:3-1} there exist $t_0\in\mathbb{R}$ such that $\mathrm{tr}\,\xi\geq 0$ and $\partial_{t}(\beta_{|\Sigma})\leq 0$ for all $t<t_0$, and $\mathrm{tr}\,\xi\leq 0$  and $\partial_{t}(\beta_{|\Sigma})\geq 0$ for all $t_0<t$,
		\item \label{thm:main:3-2} {$\overline{M}$ has moderate proper time rate over $\Sigma$},
		\item \label{thm:main:3-3} the mean curvature vector of $\Sigma$ satisfies $(\tau-t_0)\overline{g}(\partial_t,\overrightarrow H)\geq 0$ and
		\item \label{thm:main:3-4} $\Sigma$ is bounded away from the infinity.
	\end{enumerate}
	Then, $\Sigma$ is contained in a spacelike slice $F_{t_0}=\{t_0\}\times F$. Moreover, $\partial_t(\beta_{|\Sigma})=\mathrm{tr}\,\xi=0$ and either $\overrightarrow{H}$ is spacelike or $\Sigma$ is minimal in $F_{t_0}$.
\end{theorem}
\begin{proof}
	As in the previous results we can assume $n>2$, since in the case $n=2$ we can easily increase the dimension. Let us consider the function $\sigma=(\tau-t_0)^{2}$ defined over $\Sigma$. It follows that
	
	\begin{equation}
		\label{eq:30}
		\tilde{\Delta} \sigma = 2 \left( \tilde{g}(\tilde{\nabla}\tau,\tilde{\nabla}\tau) + (\tau-t_0)\tilde{\Delta}{(\tau-t_0)} \right).
	\end{equation}
	Now observe that, from (\Ref{eq:20}) and assumptions~\ref{thm:main:3-1} and~\ref{thm:main:3-3} of the theorem we conclude that $(\tau-t_0)\tilde{\Delta} \tau$ is always non-negative. In particular, $\sigma$ is a subharmonic function defined over $\Sigma$. From assumption~\ref{thm:main:3-4} $\tau$ is bounded both from above and from below, so it is $\sigma$. Thus, it should be constant by the parabolicity of $\Sigma$ and the result follows as in Theorem~\ref{thm:rig_general}.
\end{proof}

{Let us observe that assumption~\ref{thm:main:3-3} of Theorem~\ref{thm:main:3} is satisfied in the particular case where $\Sigma^n$ is a maximal submanifold, whereas assumtpion~\ref{thm:main:3-1} generalizes the hypothesis of $\overline{M}$ having a local contracting phase change at $t_0$ (see Section \ref{sec:orthogospace} for the definition of orthogonally splitted spacetimes with a local phase change).}

\section{On trapped surfaces in orthogonally splitted spacetimes}
\label{sec:trapped}
As an application of the previous results, we will focus on this section on the study of trapped surfaces. Although we will restrict to the standard concepts of trapped surfaces in $4$-dimensional physically realistic spacetimes, let us remark that all the results in this section will be easily extensible for submanifolds of any codimension in a spacetime of arbitrary dimension.

Let us begin by introducing the basic concepts regarding trapped surfaces that we will require for this section (see \cite{Kriele} and \cite{MarsSeno} for details):

\begin{definition}
	\label{def:main:2}
	Let $(\overline{M}^4,\overline{g})$ be a $4$-dimensional spacetime, and consider $\Sigma^2$ a surface of $\overline{M}$ with mean curvature vector $\Hm$. Then, we will say that $\Sigma$ is:
	
	\begin{enumerate}[label=(\roman*)]
		\item\label{def:ftr1} \emph{future trapped} if $\Hm$ is timelike and future pointing all over $\Sigma$.
		\item\label{def:ftr2} \emph{nearly future trapped} if $\Hm$ is causal and future pointing all over $\Sigma$, and timelike at at least one point of $\Sigma$.
		\item\label{def:ftr3} \emph{weakly future trapped} if $\Hm$ is causal, future pointing and non-zero at at least one point of $\Sigma$.
		\item\label{def:ftr4} \emph{marginally future trapped} if $\Hm$ is lightlike, future pointing and non-zero at at least one point of $\Sigma$.
		\item\label{def:ftr5} \emph{extremal} if $\Hm=0$ all over $\Sigma$.
	\end{enumerate}
	The \emph{past} analogous notions of~\ref{def:ftr1},~\ref{def:ftr2},~\ref{def:ftr3} and~\ref{def:ftr4} are defined in a similar way considering $\Hm$ to be past pointing.
\end{definition}

As we can see, if $\Sigma$ is (nearly, weakly, marginally) future trapped or extremal, then $g(\partial_t,\Hm)\leq 0$, and so conditions \ref{thm:rig_general-3} in Theorem~\ref{thm:rig_general} and~\ref{thm:non-ex-min-3} in Theorem \ref{thm:non_ex_min} are satisfied. Moreover, condition~\ref{thm:main:3-3} in Theorem~\ref{thm:main:3} is immediately satisfied for extremal surfaces. Hence, such theorems have some nice interpretations in terms of trapped surfaces. For instance, as a direct consequence of Theorem~\ref{thm:rig_general} we obtain a first non-existence result for parabolic trapped surfaces.

\begin{theorem}
	\label{thm:main:1}
	{Let $(\overline{M}^4,\overline{g})$ be a non-contracting (resp. non-expanding) orthogonally splitted spacetime with moderate proper time rate. Then, there is no parabolic (nearly, weakly, marginally) future (resp. past) trapped surface in $\overline{M}$ bounded away from the past (resp. future) infinity. Furthermore, if it exists any parabolic extremal surface in $\overline{M}$ bounded away from the past (resp. future) infinity, it is contained in a slice $F_{t_0}$.}
\end{theorem}

\begin{remark}
	{In fact, Theorem~\ref{thm:main:1} could be stated under a slightly weaker assumption. It is not necessary to ask $\overline{M}$ to have moderate proper time rate, but just to have moderate proper time rate over any possible surface satisfying the conditions of the theorem.}
\end{remark}

Observe that Theorem~\ref{thm:main:1} provides three conditions to ensure the non-existence of any kind of trapped surface: monotonicity, parabolicity and moderate proper time. If we consider the case where $\beta\equiv 1$, the latter is directly satisfied, and {the monotonicity of a certain physically realistic subset of $\overline{M}$} can be provided by imposing some initial condition at one of the slices $F_{t_0}$ and a (sectional) curvature condition for timelike planes (in the spirit of Proposition \ref{auxiliarybeta1}). In this sense, from Proposition~\ref{auxiliarybeta1} and Theorem~\ref{thm:main:1} we can easily deduce the following result.

\begin{corollary}\label{cor:6.7} 
	Let $(\overline{M}^4,\overline{g})$ be an orthogonally splitted spacetime as in~\eqref{eq:1} with $\beta\equiv 1$. Assume that, for some $t_0$, the shape operator of the slice $F_{t_0}$ is positive (resp. negative) semi-definite, and that the sectional curvatures of {timelike planes of} the spacetime are non-negative. Then, there are no parabolic (nearly, weakly, marginally) future (resp. past)  trapped surfaces contained in $I^+(F_{t_0})$ (resp. $I^-(F_{t_0})$). Furthermore, if it exists any parabolic extremal surface in $I^+(F_{t_0})$ (resp. $I^-(F_{t_0})$), it is contained in a slice.
\end{corollary}

Regarding the parabolicity, and as mentioned in Section~\ref{sec:parabolicity}, in the above non-existence results we can replace such a condition by completeness  {and non-negativeness of the Gaussian curvature function, or more generally by} the finite total curvature condition. In fact, we also have the following result:

\begin{corollary}
	Under the hypothesis of Theorem \ref{thm:main:1}, there are no (nearly, weakly, marginally) future (resp. past) trapped surface bounded  {away from the past (resp. future) infinity} with finite  {total} curvature. In particular, there is no future (resp. past) trapped surface (of any kind), bounded away from the past (resp. future) infinity with  {non-}negative Gaussian curvature  {outside a compact set.} 
\end{corollary}

However, even in the case where we cannot ensure parabolicity, Theorem~\ref{thm:non_ex_min} allows us to give some information regarding the shape of trapped surfaces. In fact, we can present the following immediate consequence of Theorem~\ref{thm:non_ex_min}.

\begin{theorem}\label{trappedMinMax}
	Let $(\overline{M}^4,\overline{g})$ be a non-contracting (resp. non-expanding) orthogonally splitted spacetime. Then, for no (nearly, weakly, marginally) future (resp. past) trapped surface, the function $\tau$ achieves a strict local minimum (resp. maximum).
\end{theorem}

\begin{remark}\label{rmk:non-exist-extremal}
	{In all the above results we have considered the spacetime to be non-contracting or non-expanding, or we have considered some conditions which guarantee this fact. Observe that if we ask the stronger hypothesis that the spacetime is expanding or contracting, we can also conclude the non-existence of extremal surfaces satisfying the same assumptions.}
\end{remark}

{We can also obtain a nice direct consequence of Theorem~\ref{thm:main:3} in the case of extremal surfaces.}

\begin{corollary}\label{cor:ext_phase_change}
	{Let $(\overline{M}^4,\overline{g})$ be an orthogonally splitted spacetime having a local contracting phase change at a certain $t_0\in\mathbb{R}$. Then, any parabolic extremal surface $\Sigma$ in $\overline{M}$ which is bounded away from the infinity and such that $\overline{M}$ has moderate proper time rate over it is contained in the slice $F_{t_0}$.}
\end{corollary}

\begin{remark}\label{rmk:closed}
	{Although a parabolic manifold is non-compact by definition, as it was remarked in Section~\ref{sec:parabolicity} the notion of parabolicity can be seen as an extension of compactness. In fact, compact manifolds satisfy the parabolicity condition, i.e. any subharmonic function on a compact manifold is necessarily constant. Therefore, all the above results also hold for the case of closed trapped surfaces. Furthermore, it is obvious that any orthogonally splitted spacetime has moderate proper time rate over any closed submanifold.}
\end{remark}

\subsection{Application to cosmological models of interest}
\label{sec:aplicaCos}
As a final section of this paper, we would like to show how the previous results translate when considering different families of spacetimes with physical interest. Our aim is twofold: on the one hand, we want to show that some of the hypotheses in our previous results are simplified for certain specific spacetimes or they are expressed in a more natural way. On the other hand, we will relate our results with some known results on the existing literature.

\smallskip

Let us begin by considering the so-called \emph{standard static spacetimes}. In the notation of this paper, a standard static spacetime is an orthogonally splitted spacetime $(\overline{M},\overline{g})$ as in \eqref{eq:1} where neither $\beta$ nor $g_t$ have dependence on the variable $t$. Therefore, the metric takes the form \[\overline{g}=-\beta dt^2+g,\] where $\beta\in C^{\infty}(F)$ is a positive function and $(F,g)$ is a Riemannian manifold.

From its definition, we can see that any standard static spacetime $(\overline{M},\overline{g})$ is both, a non-expanding and a non-contracting orthogonally splitted spacetime. Moreover, the moderate proper time rate condition can be directly satisfied if we assume for instance the compactness of the surface. Hence, our technique allows us to recover some known results in~\cite{MarsSeno} and in~\cite{Pel}. For instance,

\begin{corollary} {\cite[Proposition 2]{MarsSeno}}
	In a standard static spacetime $(\overline{M}^{4},\overline{g})$, if there exists a extremal closed spacelike surface, it must be  {a minimal surface of} a slice $F_{t_0}$. 
\end{corollary}
In a similar way, in the context of standard static spacetimes, Theorem \ref{trappedMinMax} allows us to obtain \cite[Theorem 4]{Pel}.

\medskip

We can perform a similar study by considering a different family included in the so-called \emph{twisted} spacetimes. In this case, we will consider $(\overline{M},\overline{g})$ an orthogonally splitted spacetime as in \eqref{eq:1} where $\beta=1$ and the dependence on $t$ of $g_t$ is encoded by a smooth function on $\overline{M}$. That is,
\begin{equation}
	\label{eq:33}
	\overline{g}=-dt^2+f\,g_0,
\end{equation}
where $f\in C^{\infty}( {\overline{M}})$ is a positive function and $(F,g_0)$ is a Riemannian manifold. Observe that this family includes the well-known \emph{Generalized Robertson-Walker spacetimes}, which are a subfamily where the function $f$ has dependence only on the time variable.

The non-contracting (resp. non-expanding) behaviour of this family of spacetimes is then reduced to the condition $\partial_tf\geq 0$ (resp. $\partial_tf\leq 0$), while the moderate proper time rate assumption is always satisfied. Hence, Theorem \ref{thm:main:1} reads:

\begin{corollary}\label{cor:GRW}
	Let $(\overline{M}^{4},\overline{g})$ be an orthogonally splitted spacetime where $\overline{g}$ takes the form in \eqref{eq:33}. If $\partial_tf\geq 0$ (resp. $\partial_tf\leq 0$), there is no parabolic (nearly, weakly, marginally) future (resp. past) trapped surface in $\overline{M}$ bounded away from the past (resp. future) infinity. Moreover, if there exists some parabolic extremal surface in $\overline{M}$, it is contained in a slice $F_{t_0}$. 
\end{corollary}
In the particular case of GRW spacetimes, the conditions for $f$ simply read as $f'\geq 0$ and $f'\leq 0$. Hence, the previous corollary generalizes some known results regarding trapped surfaces in Generalized Robertson-Walker spacetimes such as~ {\cite[Result 5.1]{IB}},~\cite[Corollary 5.2]{ACC} and the non-existence results contained in~\cite[Section 4.1]{Pel2}.  {Similarly, Theorem~\ref{trappedMinMax} applied to GRW spacetimes generalizes~\cite[Result 5.2]{IB} and~\cite[Theorem 3]{Pel2}.}

\medskip

As a final example of interest, we will consider a family included in the doubly twisted spacetimes. Concretely, we will consider $(\overline{M},\overline{g})$ an orthogonally splitted spacetime of the form:
\begin{equation}\label{sphericallysymmetric}
	\overline{M}=I\times J\times F,\qquad \overline{g}=-\beta dt^2 + f_1dr^2 + f_2 g_0,	
\end{equation}
where $I,J$ are real intervals, $f_1,f_2:\overline{M}\rightarrow \mathbb{R}$ are smooth positive functions and $(F,g_0)$ is a $2$-dimensional Riemannian surface. Observe that if $f_1,f_2\in C^{\infty}(J)$, such a metric is the classical expression for an arbitrary static and spherically symmetric spacetime in appropriate coordinates (see \cite{Wald}).

As in the previous family of spacetimes, we can easily encode the non-contracting and non-expanding character of $(\overline{M},\overline{g})$ in terms of the derivatives of the functions $\beta,f_1,f_2$ with respect to the time coordinate, obtaining again a similar result as in Corollary \ref{cor:GRW}  {for trapped surfaces in a doubly twisted spacetime as in~\eqref{sphericallysymmetric}}. However, we would like to focus our attention now on the parabolic assumption. Concretely, we would like to make use of the concepts of isocausality and quasi-isometries developed in Section~\ref{sec:parabolicity}  {to find a condition that ensures the parabolicity of a surface given as a graph as follows}. {Let us consider a fixed $r_0\in J$ and any smooth function $u:F\rightarrow I$ and, similarly as in~\eqref{eq:24}, let us define $\Sigma_{r_0,u}$ by}
\begin{equation}\label{eq:sigma_r0u}
	{\Sigma_{r_0,u}=\left\{(u(x),r_0,x)\in\overline{M}:x\in F\right\},}
\end{equation}
We can then state the following result.

\begin{corollary}\label{cor:twisted}
	{Let $(\overline{M}^4,\overline{g})$ be a non-contracting (resp. non-expanding) orthogonally splitted spacetime as in \eqref{sphericallysymmetric}. Let us assume that there exist $\alpha_1,\alpha_2\in C^{\infty}(I)$ positive functions such that}
	\begin{equation}
		\label{eq:2}
		{\alpha_2\leq \dfrac{f_2}{\beta}\leq \alpha_1}.
	\end{equation}
	{Then, there is no $r_0\in J$ and no function $u:F\rightarrow I$  bounded from below (resp. from above) such that:}
	\begin{enumerate}[label=(\roman*)]
		\item\label{cor:twisted-1} {$\alpha_1(u)$, $\alpha_2(u)$ and $\beta|_{\Sigma_{r_0,u}}$ are bounded functions and bounded away from zero, i.e. with positive infimums,}
		\item\label{cor:twisted-2} {$\left|\nabla^0 u\right|_0\leq k\alpha_2$ for some $0<k<1$, $\nabla^0$ and $|\cdot |_0$ being the gradient operator and the norm in $(F,g_0)$, and} 
		\item\label{cor:twisted-3} {$\Sigma_{r_0,u}$ is a (nearly, weakly, marginally) future (resp. past) trapped surface in $(\overline{M},\overline{g})$.}		
	\end{enumerate}
	\smallskip
	{Furthermore, if it exists $r_0\in J$ and a function $u:F\rightarrow I$ bounded from below (resp. from above) such that~\ref{cor:twisted-1} and~\ref{cor:twisted-2} are satisfied and $\Sigma_{r_0,u}$ is a extremal surface, then $u$ must be constant, i.e., $\Sigma_{r_0,u}=\left\{t_0\right\}\times\left\{r_0\right\}\times F$ for a certain $t_0\in I$.}
\end{corollary}
\begin{proof}
	Our aim is to prove that $\Sigma_{r_0,u}$ is parabolic, so that Theorem \ref{thm:main:1} can be applied. To that end we will apply Proposition \ref{prop:main:2}. Conditions \ref{prop:main:2-2}, \ref{prop:main:2-4} and \ref{prop:main:2-3} of such proposition are directly satisfied from the hypothesis of the corollary, while assumption \ref{prop:main:2-1} follows from~\eqref{eq:2}, taking into account that $g_t=f_2g_0$ and recalling Remark \ref{rem:main:1}.
\end{proof}

\begin{remark}
	Corollary~\ref{cor:twisted}, even if it appears to be quite technical, allows us to show that the non-existence of trapped surfaces (of any kind) is invariant under \emph{time-controlled perturbations} on the warping function. Let us explain this more clearly by assuming that $\beta\equiv 1$. The result ensures, in particular, that any non-contracting $(\overline{M},\overline{g})$ of the form \eqref{sphericallysymmetric} with $\alpha_2\leq f_2\leq \alpha_1$, being $\alpha_1,\alpha_2$ bounded away from zero, has no future trapped surface of the form $\Sigma_{r_0,u}$ bounded away from the past infinite (an analogous result for past trapped surfaces is obtained) unless $\alpha_i(u)$ are unbounded (in particular, the image of $u$ cannot be a pre-compact set).	
\end{remark}

\section*{Acknowledgements}

The authors would like to express their heartfelt gratitude to the referees for their valuable comments, which have significantly enhanced the quality of the manuscript. All authors are partially supported by the Spanish MCIN Grant PID2021-126217NB-I00 with FEDER funds and by the Regional Government of Andalusia ERDEF Project P20-01391.

\medskip

\bibliographystyle{amsplain}

\begin{thebibliography}{10}
	\bibitem{Ahl} L. V. Ahlfors, Sur le type d´une surface de Riemann, \emph{C.R. Acad. Sci. Paris} \textbf{201} (1935),30--32.
	\bibitem{AGM} J.A. Aledo, J.A. G\'alvez and P. Mira, Marginally trapped surfaces in $\mathbb{L}^4$ and an extended Weierstrass-Bryant representation, \emph{Ann. Global Anal. Geom} \textbf{28}(4) (2005), 395--415.
	\bibitem{ARS19} J. A. Aledo, R. M. Rubio and J. J. Salamanca, Compact maximal hypersurfaces in globally hyperbolic spacetimes, \emph{Class. Quantum Grav.} \textbf{36} (1) (2019), 015017 (14pp).
	\bibitem{ACC} L. J. Al\'ias, V. L. C\'anovas and A. G. Colares, Marginally trapped submanifolds in generalized Robertson-Walker spacetimes, \emph{Gen. Relativ. Gravit.} \textbf{49} (2017), 23pp. 
	\bibitem{AMR} L. J. Al\'ias, P. Mastrolia and M. Rigoli, \emph{Maximum principles and geometric applications}, Springer, (2016).
	
	
	\bibitem{AMS2008} L. Andersson, M. Mars, W. Simon, Stability of marginally outer trapped surfaces and existence of marginally outer trapped tubes, \emph{Adv. Theor. Math. Phys.} \textbf{12}(4), 853--888.
	
	\bibitem{AM} L. Andersson and J. Metzger, The Area of Horizons and the Trapped Region, \emph{Commun. Math. Phys.} \textbf{290} (2009), 941--972.
	\bibitem{IB} I. Bengtsson and J.M. M. Senovilla, Region with trapped surfaces in spherical symmetry, its core and their boundaries, \emph{Phys. Rev. D} \textbf{83} (2011), 044012.
	\bibitem{BS05} A. N. Bernal and M. Sánchez, Smoothness of time functions and the metric splitting of globally hyperbolic spacetimes, \emph{Commun. Math. Phys.} \textbf{257} (2005), 43--50.
	\bibitem{BS06} A. N. Bernal and M. Sánchez, Further results on the smoothability of Cauchy hypersurfaces and Cauchy time functions, \emph{Lett. Math. Phys.} \textbf{77} (2) (2006), 183--197.
	\bibitem{Cr} P.T. Chru\'sciel, G.J. Galloway and E. Ling, Weakly trapped surfaces in asymptotically de Sitter spacetimes, \emph{Class. Quantum Grav.} \textbf{35} (2018), 135001 (9pp).
	
	\bibitem{Fl} J.L. Flores, S. Haesen and M. Ortega, New examples of marginally trapped surfaces and tubes in warped spacetimes, \emph{Class. Quantum Grav.} \textbf{27} (2010), 145021 (18pp).
	\bibitem{FHS1} J. L. Flores, J. Herrera and M. Sánchez, Computability of the causal boundary by using
	isocausality, \emph{Class. Quantum Grav.}, \textbf{30} (7) (2014), 075009.
	
	\bibitem{GS} G. Galloway and J.M. Senovilla, Singularity theorems based on trapped submanifolds
	of arbitrary co-dimension, \emph{Class. Quantum Grav.}, \textbf{27} (15) (2010), 152002.
	
	\bibitem{GP-Sen} A. García-Parrado and J. M. M. Senovilla, Causal relationship: a new tool for the causal characterization of Lorentzian manifolds, \emph{Class. Quantum Grav.} \textbf{22} (2003), 625--664.
	
	\bibitem{Gr} A. Grigor'yan, Analytic and geometric background of recurrence and non-explosion of the Brownian motion on Riemannian manifolds, \emph{B. Am. Math. Soc.}, \textbf{36} (1999), 135--249.
	
	\bibitem{HE} S.W. Hawking and G. Ellis, \emph{The large scale structure of space-time,}
	Cambridge Monographs on Mathematical Physics, No. 1. Cambridge University Press, London-New York, 1973.
	
	\bibitem{HP} S.W. Hawking and R. Penrose, The singularities of gravitational collapse and cosmology, \emph{Proc. Roy. Soc. Lond. A.}
	\textbf{314}  (1970), 529--548.
	
	\bibitem{Hub} A. Huber, On subharmonic functions and differential geometry in the large, \emph{Comment. Math. Helv.} \textbf{32} (1957), 13--72.
	
	\bibitem{Ka} M. Kanai, Rough isometries and the parabolicity of Riemannian manifolds, \emph{J. Math. Soc. Jpn.}, \textbf{38} (1986), 227--238.
	
	\bibitem{Kaz} J. L. Kazdan, Parabolicity and the Liouville property on complete Riemannian manifolds, \emph{Aspects of Math.} \textbf{10} (1987), 153--166.
	
	\bibitem{Kriele} M. Kriele, Spacetime: Foundations of General Relativity and Differential Geometry, \emph{Lecture Notes in Physics Monographs}, Springer Berlin Heidelberg (2003).  
	
	\bibitem{Lee} P. Lee,  Curvature and function theory on Riemannian manifolds, 
	\emph{Surveys in Differential Geometry}, {\bf{Vol. II}}, International Press (2000), 375--432.
	
	\bibitem{MarsSeno} M. Mars and J. M. M. Senovilla, Trapped surfaces and
	symmetries. \emph{Class. Quantum Grav.}, \textbf{20} (24) (2003),
	293–300.
	
	\bibitem{Mu} E. Musso and L. Nicolodi, Marginally outer trapped surfaces in de Sitter space by low-dimensional geometries, \emph{ J. Geom. Phys.} \textbf{96} (2015), 168--186.
	
	\bibitem{Pel} J. A. S. Pelegr\'in, Minimal and weakly trapped submanifolds in standard static spacetimes, \emph{J. Math. Anal. Appl.} \textbf{480} (2019), 123448, 10pp.
	
	\bibitem{Pel2} J. A. S. Pelegr\'in, Extremal and weakly trapped submanifolds in generalized Robertson-Walker spacetimes, \emph{J. Math. Anal. Appl.} \textbf{505} (2022), 125533, 12pp.
	
	\bibitem{Pe} R. Penrose, Gravitational collapse and space-time singularities, \emph{Phys. Rev. Lett.} \textbf{14}(3) (1965),  57.
	
	\bibitem{Petersen} P. Petersen, Riemannian Geometry, \emph{Springer New York} (1998).
	
	\bibitem{sanchez2021} M. Sánchez, Globally hyperbolic spacetimes: slicings,
	boundaries and counterexamples. \emph{Gen. Relativ. Gravit.} \textbf{54} (124) (2022). 
	
	\bibitem{Sen} J.M.M. Senovilla, Trapped surfaces, \emph{Internat. J. Modern Phys. D} \textbf{20} (2011),  2139--2168
	
	
	\bibitem{Wald} R. Wald, General Relativity, \emph{University of Chicago Press} (1984).
\end{thebibliography}

\end{document}